\documentclass[]{paper}

\usepackage{epstopdf}
\usepackage[caption=false]{subfig}
\usepackage{algorithm}
\usepackage{algorithmic}
\usepackage[numbers,sort&compress]{natbib}
\bibpunct[, ]{[}{]}{,}{n}{,}{,}
\renewcommand\bibfont{\fontsize{10}{12}\selectfont}

\theoremstyle{plain}
\newtheorem{theorem}{Theorem}[section]
\newtheorem{assumption}{Assumption}
\newtheorem{lemma}[theorem]{Lemma}
\newtheorem{corollary}[theorem]{Corollary}

\theoremstyle{definition}
\newtheorem{definition}[theorem]{Definition}

\theoremstyle{remark}

\begin{document}


\title{Stochastic Trust-Region Methods for Over-parameterized Models}


\author{
Aike Yang\textsuperscript{a} \and
Hao Wang\textsuperscript{a} \\
\textsuperscript{a}ShanghaiTech University, Shanghai, China\footnote{This is a preprint.}
}

\maketitle

\begin{abstract}
Under interpolation-type assumptions such as the strong growth condition, stochastic optimization methods can attain convergence rates comparable to full-batch methods, but their performance—particularly for SGD—remains highly sensitive to step-size selection. To address this issue, we propose a unified stochastic trust-region framework that eliminates manual step-size tuning and extends naturally to equality-constrained problems. For unconstrained optimization, we develop a first-order stochastic trust-region algorithm and show that, under the strong growth condition, it achieves an iteration and stochastic first-order oracle complexity of $O(\varepsilon^{-2}\log(1/\varepsilon))$ for finding an $\varepsilon$-stationary point. For equality-constrained problems, we introduce a quadratic-penalty-based stochastic trust-region method with penalty parameter $\mu$, and establish an iteration and oracle complexity of $O(\varepsilon^{-4}\log(1/\varepsilon))$ to reach an $\varepsilon$-stationary point of the penalized problem, corresponding to an $O(\varepsilon)$-approximate KKT point of the original constrained problem. Numerical experiments on deep neural network training and orthogonally constrained subspace fitting demonstrate that the proposed methods achieve performance comparable to well-tuned stochastic baselines, while exhibiting stable optimization behavior and effectively handling hard constraints without manual learning-rate scheduling.

\end{abstract}


\section{Introduction}

Stochastic optimization plays a fundamental role in modern machine learning, data science, and large-scale scientific computing.
A canonical example is the empirical risk minimization (ERM) problem
\[
\min_x f(x)=\frac{1}{n}\sum_{i=1}^{n}f_i(x),
\]
where $n$ denotes the number of data samples and $f_i(x)$ represents the loss associated with the $i$-th observation.
This formulation encompasses a wide range of applications, including supervised learning, inverse problems, and parameter estimation in large-scale models
\cite{vapnik1999overview,bottou2018optimization,shalev2014understanding}.

Among stochastic optimization methods, stochastic gradient descent (SGD) and its variants are arguably the most widely used due to their simplicity and low per-iteration cost.
By approximating the full gradient using randomly sampled data points or mini-batches, SGD significantly reduces computational overhead compared to deterministic gradient methods
\cite{robbins1951stochastic,bottou2010large}.
The convergence behavior of SGD is classically analyzed under the bounded variance assumption, which postulates that the variance of the stochastic gradient estimator is uniformly bounded
\cite{nemirovski2009robust,Ghadimi2013SGD}.
While this assumption enables convergence guarantees in both convex and nonconvex settings, it can be overly conservative for modern over-parameterized models.

Recent advances in stochastic optimization theory have emphasized the importance of interpolation and the strong growth condition (SGC) as refined characterizations of gradient noise in stochastic optimization
\cite{schmidt2018minimizing,vaswani2019fast}.
The interpolation condition asserts that any stationary point of the empirical objective is simultaneously a stationary point of each individual loss function, namely,
\[
\nabla f(x)=0 \Rightarrow \nabla f_i(x)=0, \quad \forall i.
\]
This property is frequently observed in over-parameterized learning models, including deep neural networks trained to zero training error
\cite{zhang2016understanding,belkin2019reconciling}.

Under interpolation, the stochastic gradient often satisfies the strong growth condition
\begin{equation*}
 \mathbb{E}\bigl[\| g_k \|^2\bigr] \le \rho \|\nabla f(x_k)\|^2,
\end{equation*}
where $g_k$ denotes the stochastic gradient at iteration $k$ and $\rho \ge 1$ is a constant.
The SGC implies interpolation and provides a sharper control of gradient noise compared to bounded variance assumptions
\cite{vaswani2019fast}.
Notably, under the SGC, SGD with a constant step size can achieve convergence rates comparable to deterministic gradient descent, even in nonconvex optimization
\cite{vaswani2019fast,zhou2019sgd}.

Despite these developments, most existing results under interpolation-based assumptions focus on SGD-type methods.
By contrast, comparatively little attention has been devoted to incorporating such assumptions into stochastic trust-region algorithms.
Trust-region (TR) methods constitute a classical and powerful framework in deterministic nonlinear optimization
\cite{conn2000trust,nocedal2006numerical}.
At each iteration, a TR method constructs a local quadratic model and restricts the trial step to a region where the model is deemed reliable, accepting the step only if the ratio of actual to predicted reduction exceeds a prescribed threshold.
This mechanism endows TR methods with strong robustness and global convergence guarantees, even in nonconvex settings.

Extending trust-region methods to stochastic optimization introduces substantial challenges.
The local model is built from noisy gradient or Hessian information, and the acceptance ratio becomes a random quantity, complicating both algorithmic design and theoretical analysis
\cite{chen2018stochastic,curtis2022fully}.
Several stochastic trust-region (STR) algorithms have been proposed in recent years, primarily under bounded variance assumptions
\cite{curtis2022fully,wang2019stochastic}.
However, the interplay between STR methods and interpolation-based conditions such as the SGC remains largely unexplored.

The problem becomes even more challenging in the presence of constraints.
Constrained stochastic optimization arises naturally in many applications, including constrained learning formulations, signal processing, and scientific computing
\cite{bertsekas1997nonlinear,recht2013parallel}.
In this work, we consider equality-constrained problems of the form
\[
\begin{aligned}
& \min_{x \in \mathbb{R}^d}\; f(x)=\frac{1}{n}\sum_{i=1}^{n} f_i(x), \\
& \text{s.t.} \quad c(x)=0,
\end{aligned}
\]
where $c:\mathbb{R}^d \to \mathbb{R}^m$ denotes a (possibly nonlinear) constraint function.

A common strategy for handling equality constraints is the use of quadratic penalty or augmented Lagrangian methods
\cite{nocedal2006numerical,bertsekas2014constrained}.
While these approaches are well understood in deterministic optimization, their stochastic counterparts often require the penalty parameter to scale as $\mu = O(1/\epsilon)$ in order to achieve $\epsilon$-feasibility, which can significantly deteriorate complexity bounds
\cite{xu2021first,boob2023stochastic}.
Existing analyses of penalty-based stochastic methods largely focus on first-order schemes, and their integration with stochastic trust-region frameworks remains limited.

In this paper, we develop a unified stochastic trust-region framework under interpolation-based assumptions for both unconstrained and constrained optimization problems.
Our main contributions are summarized as follows:
\begin{itemize}
    \item For the unconstrained case, we analyze a first-order stochastic trust-region algorithm under the strong growth condition and establish an iteration and stochastic first-order oracle complexity of $O(\epsilon^{-2}\log(1/\epsilon))$ to reach $\|\nabla f(x)\| \le \epsilon$.

    \item For the constrained case, we propose a penalty-based stochastic trust-region method.
    By selecting the penalty parameter as $\mu = 1/\epsilon$, we prove an iteration and oracle complexity of $O(\epsilon^{-4}\log(1/\epsilon))$ to obtain an $O(\varepsilon)$-approximate KKT point of the original constrained problem.

    \item We validate the proposed methods through numerical experiments on nonconvex deep learning tasks and orthogonally constrained subspace fitting problems.
    The results demonstrate that stochastic trust-region methods are competitive with standard stochastic optimizers in unconstrained settings, while providing a principled and flexible framework for constrained stochastic optimization.
\end{itemize}

\section{Literature review}

Stochastic optimization has been extensively studied in the past decade. Early works focused on stochastic gradient descent (SGD) and its convergence under the bounded variance assumption. In particular,  \cite{Ghadimi2013SGD} established an $O(1/\sqrt{K})$ rate for nonconvex stochastic optimization, which became the benchmark complexity result. To go beyond the variance-bounded setting, interpolation and the strong growth condition (SGC) were introduced. These assumptions, frequently satisfied in over-parameterized models, allow constant step-size SGD to achieve accelerated convergence. Notably, \cite{Schmidt2013StrongGrowth} showed fast rates of SGD under SGC in convex settings, and \cite{Vaswani2019constant,vaswani2019painless} further extended such results to nonconvex problems.

Trust-region (TR) methods have a long tradition in deterministic nonlinear optimization. They are well known for global convergence and strong robustness, as surveyed in classical texts on nonlinear programming \cite{conn2000trust}. In recent years, several works have studied stochastic variants of TR methods, where the model is built from noisy gradients or Hessians. 
For example, \cite{Gratton2015STR} introduced a stochastic TR framework with probabilistic models, showing global convergence under noise; \cite{chen2018stochastic} and \cite{Cartis2018Prob} extended analysis to nonconvex and derivative-free settings; \cite{Paquette2020SCR} proposed stochastic cubic-regularization versions with improved sample complexity.  
Moreover, variance-reduction techniques have recently been incorporated into stochastic TR frameworks, as in the TR-SVR method, which combines trust-region modelling with variance-reduced gradient estimators to enhance efficiency and stability in large-scale stochastic optimization tasks \cite{Zheng2024TRSVR}. Beyond the classical Lipschitz-smooth setting, \cite{Xie2024TRGeneralized} extend stochastic trust-region methods to the more general \((L_0, L_1)\)-smoothness regime, relevant for deep learning and distributionally robust optimization, achieving first- and second-order convergence; notably, with variance reduction, their second-order TR method attains an \(\mathcal{O}(\varepsilon^{-3})\) complexity—an optimal bound in this context. Earlier, \cite{Shen2020STR} introduced a stochastic trust-region algorithm using variance-reduced Hessian estimators for nonconvex finite-sum minimization, obtaining convergence to an \((\varepsilon, \sqrt{\varepsilon})\)-approximate local minimum in \(\tilde{\mathcal{O}}(\sqrt{n}/\varepsilon^{1.5})\) stochastic Hessian oracle calls—a state-of-the-art sample complexity at the time.  
Other notable approaches include the inexact-restoration-based SIRTR method (\cite{Bellavia2023SIRTR}), which establishes convergence in probability for stochastic trust-region steps in finite-sum contexts, and STARS (\cite{DzahiniWild2022STARS}), a derivative-free trust-region framework using random subspaces with almost-sure convergence and competitive iteration complexity. Finally, \cite{Ha2024ASTRO} analyze variance-reduction enhancements in stochastic TR model-update and candidate-evaluation steps, showing dramatic improvements in almost-sure sample complexity—from \(\tilde{\mathcal{O}}(\varepsilon^{-6})\) to \(\tilde{\mathcal{O}}(\varepsilon^{-2})\) in smooth first-order settings.  
These works mainly focused on the unconstrained setting and demonstrated that stochastic TR methods can maintain favorable complexity guarantees comparable to stochastic gradient methods, while often achieving better practical robustness. However, theoretical understanding in this direction is still relatively limited compared to line-search–based methods \cite{Ghadimi2013SGD}.

The study of constrained stochastic optimization has largely centered on stochastic sequential quadratic programming (SQP) and augmented Lagrangian methods \cite{bertsekas2014constrained,Curtis2021StochasticSQP}. Penalty-based approaches, although widely used in deterministic optimization \cite{conn2000trust}, have received less attention in the stochastic context. One reason is that the choice of penalty parameter has a significant impact: ensuring feasibility typically requires setting $\mu = 1/\epsilon$, which directly influences convergence complexity \cite{Birgin2014ALM}. Recent works on stochastic min–max optimization and constrained learning have begun to address these challenges \cite{Nouiehed2019Minimax,Lin2020Minimax,Rafique2021WeaklyConvex,wang2023zeroth}, but penalty-based stochastic trust-region methods remain underexplored.

Our work contributes to filling this gap by establishing a unified analysis for unconstrained and penalty-based constrained stochastic trust-region methods.

\section{Problem Statement and Algorithm}
In this section, we present the stochastic trust-region framework studied in this paper.
We first introduce the unconstrained finite-sum stochastic optimization problem and review
a first-order stochastic trust-region method.
We then extend the framework to equality-constrained problems by means of a quadratic penalty formulation,
leading to a stochastic trust-region algorithm for constrained optimization.

\subsection{Unconstrained Stochastic Optimization}
We consider the finite-sum stochastic optimization problem
\[
\min_{x \in \mathbb{R}^d} f(x) := \frac{1}{n} \sum_{i=1}^n f_i(x),
\]
where each component function \( f_i:\mathbb{R}^d \to \mathbb{R} \) is assumed to be continuously differentiable. 
Problems of this form arise ubiquitously in machine learning and data-driven optimization, where the objective is typically accessed through inexpensive but noisy gradient evaluations of individual samples.

Rather than relying on predetermined step sizes as in classical stochastic gradient methods, we adopt a stochastic trust-region perspective, which dynamically controls the step length according to local model adequacy.
At iteration \(k\), given the current iterate \(x_k\), a local quadratic model of \(f\) is constructed as
\[
m_k^{\text{quad}}(p) = f(x_k) + g_k^\top p + \tfrac{1}{2} p^\top H_k p, 
\quad \|p\| \leq \Delta_k,
\]
where \(g_k\) denotes a stochastic estimator of \(\nabla f(x_k)\), \(H_k\) is an approximation to the Hessian, and \(\Delta_k>0\) is the trust-region radius.
The trial step \(p_k\) is obtained by approximately solving the subproblem.

To isolate the essential mechanisms of the method and facilitate complexity analysis, we first focus on a first-order variant in which curvature information is ignored by setting \(H_k = I\).
In this case, the model simplifies to
\begin{equation}
\label{subproblem 1}
    m_k(p) = g_k^\top p + \tfrac{1}{2}\|p\|^2, 
\quad \|p\| \leq \Delta_k,
\end{equation}
which admits a closed-form solution given by
\[
p_k = -a_k g_k, \quad 
a_k = 
\begin{cases}
1, & \|g_k\| \leq \Delta_k, \\[6pt]
\Delta_k / \|g_k\|, & \|g_k\| > \Delta_k.
\end{cases}
\]
Thus, the step corresponds to a truncated negative gradient direction, where the truncation is automatically enforced by the trust-region constraint.
This construction can be interpreted as an adaptive step-size rule driven by model-based acceptance criteria, rather than by predetermined schedules.

The quality of the trial step is assessed through the ratio of actual to predicted reduction, computed using the same randomly sampled component function.
Based on this ratio, the algorithm decides whether to accept the step and how to update the trust-region radius.
The resulting procedure is summarized in Algorithm~\ref{alg:stochastic-tr}.

\begin{algorithm}[h]
    \caption{First-Order Stochastic Trust-Region Method}
    \label{alg:stochastic-tr}
    \begin{algorithmic}[1]
        \REQUIRE Constants $\bar{\Delta}>0$, $\Delta_0\in(0,\bar{\Delta})$, $0<c_0\leq c_1\leq c_2<1$, and $\nu_1, \nu_2>1$.
        \FOR{$k=0,1,2,\ldots$}
            \STATE Sample an index $i \in \{1,\dots,n\}$ and compute a stochastic gradient $g_k = \nabla f_{i}(x_k)$.
            \STATE Compute a Cauchy-type step solving the trust-region subproblem:
            \[
            p_k = \arg\min_{\|p\|\leq \Delta_k} \; m_k(p).
            \]
            \STATE Evaluate the ratio of actual to predicted reduction:
            \[
            r_k = \frac{f_{i}(x_k) - f_{i}(x_k+p_k)}{m_k(0) - m_k(p_k)}.
            \]
            \IF{$r_k > c_0$}
                \STATE $x_{k+1} = x_k + p_k$.
            \ELSE
                \STATE $x_{k+1} = x_k$.
            \ENDIF
            \STATE Update the trust-region radius:
            \[
            \Delta_{k+1} =
            \begin{cases}
                \Delta_k / \nu_1, & r_k < c_1, \\[0.8ex]
                \min(\nu_2 \Delta_k, \bar{\Delta}), & r_k > c_2, \\[0.8ex]
                \Delta_k, & \text{otherwise.}
            \end{cases}
            \]
        \ENDFOR
    \end{algorithmic}
\end{algorithm}

\subsection{Equality-Constrained Optimization via Quadratic Penalty}
In many practical scenarios the decision variables are required to satisfy structural constraints,
which motivates extending the proposed stochastic trust-region framework to constrained settings.
In this work, we restrict our attention to equality-constrained problems of the form
\begin{equation}
\label{eq:constraint}
    \begin{aligned}
    & \min_{x \in \mathbb{R}^d} \; f(x)=\frac{1}{n}\sum_{i=1}^{n} f_i(x), \\
    & \text{s.t.} \quad c(x)=0,
    \end{aligned}
\end{equation}
where \( c:\mathbb{R}^d \to \mathbb{R}^m \) is continuously differentiable.
Focusing on equality constraints allows us to clearly illustrate the interaction between stochasticity,
trust-region mechanisms, and constraint handling, while keeping the algorithmic structure tractable.

Directly enforcing feasibility within a stochastic trust-region framework can be challenging due to noisy gradient information.
To circumvent this difficulty, we adopt a quadratic penalty approach, transforming the constrained problem into a sequence of unconstrained ones.
Specifically, we consider the penalized objective
\[
\phi(x) := f(x)+\frac{\mu}{2}\|c(x)\|^{2},
\]
where \(\mu>0\) is a fixed penalty parameter that controls the trade-off between objective minimization and constraint satisfaction.

Accordingly, we define the per-sample penalized function
\[
\varphi_i(x) := f_i(x) + \frac{\mu}{2}\|c(x)\|^2,
\]
and construct at iteration \(k\) the stochastic gradient and Hessian approximation
\[
g_k := \nabla f_{i}(x_k) + \mu \nabla c(x_k)^{\top} c(x_k), 
\qquad
H_k := I + \mu \nabla c(x_k)^{\top}\nabla c(x_k),
\]
where \(H_k\) is symmetric positive definite and captures curvature information induced by the constraint at negligible additional cost.

The resulting quadratic model is
\[
m_k^{\phi}(p)= g_k^{\top} p+\frac{1}{2}p^{\top}H_kp, 
\qquad \|p\|\leq \Delta_k,
\]
whose minimizer over the trust region admits a Cauchy-type closed form
\[
p_k = -a_k g_k, \quad 
a_k =
\begin{cases} 
\dfrac{\|g_k\|^2}{g_k^{\top}H_kg_k}, & \|p_k\| < \Delta_k, \\[1.2ex]
\dfrac{\Delta_k}{\|g_k\|}, & \|p_k\|=\Delta_k.
\end{cases}
\]

As in the unconstrained case, the step length is automatically regulated by the trust-region radius, while the penalty term steers the iterates toward feasibility.

The complete stochastic trust-region method with quadratic penalty is summarized in Algorithm~\ref{alg:penalty-tr}.
Throughout the paper, we measure complexity in terms of stochastic first-order oracle calls.
Since Algorithms~\ref{alg:stochastic-tr} and~\ref{alg:penalty-tr} sample exactly one component gradient per iteration,
iteration complexity coincides with oracle complexity, a convention that will be adopted in the subsequent analysis.

\begin{algorithm}[h]
	\caption{Stochastic Trust-Region Method with Quadratic Penalty}
	\label{alg:penalty-tr}
	\begin{algorithmic}[1]
		\REQUIRE Constants $\bar{\Delta}>0$, $\Delta_0\in(0,\bar{\Delta})$, $0<c_0\leq c_1\leq c_2<1$, and $\nu>1$.
		\FOR{$k=0,1,2,\ldots$}
		    \STATE Sample an index $i \in \{1,\dots,n\}$ and compute $\nabla f_{i}(x_k)$.
		    \STATE Form $g_k = \nabla f_{i}(x_k) + \mu \nabla c(x_k)^{\top} c(x_k)$ and $H_k = I + \mu \nabla c(x_k)^{\top}\nabla c(x_k)$.
		    \STATE Compute a Cauchy-type step solving the trust-region subproblem: 
		    \[
		    p_k = \arg\min_{\|p\|\leq \Delta_k} m_k^{\phi}(p).
		    \]
		    \STATE Evaluate the ratio
		    \[
		    r_k = \frac{\varphi_{i}(x_k) - \varphi_{i}(x_k + p_k)}{m_k^{\phi}(0) - m_k^{\phi}(p_k)}.
		    \]
		    \IF{$r_k > c_0$}
		        \STATE $x_{k+1} = x_k + p_k$.
		    \ELSE
		        \STATE $x_{k+1} = x_k$.
		    \ENDIF
		    \STATE Update $\Delta_{k+1}$ as in Algorithm~\ref{alg:stochastic-tr}.
		\ENDFOR
	\end{algorithmic}
\end{algorithm}

\section{Global Convergence Analysis}

In this section we study the global convergence properties of the proposed stochastic trust-region algorithms. 
We first analyze the unconstrained method (Algorithm~\ref{alg:stochastic-tr}) and then extend the analysis to the equality-constrained variant (Algorithm~\ref{alg:penalty-tr}). 
The latter algorithm can be viewed as a natural extension of the former through a quadratic penalty formulation.

\subsection{Convergence Analysis for the Unconstrained Problem}

In this subsection we analyze the global convergence of
Algorithm~\ref{alg:stochastic-tr}.
To establish the theoretical results, we first introduce several
standard assumptions used throughout the analysis of the unconstrained problem.

We first assume that the stochastic gradient estimator is unbiased.
\begin{assumption}
\label{assu:unbiased_uc}
For all $k \in \mathbb{N}$, the stochastic gradient estimator $g_k$ satisfies
\[
\mathbb{E}_k[g_k] = \nabla f(x_k).
\]
\end{assumption}

This assumption ensures that the stochastic gradient is an unbiased estimator of the true gradient,
which is standard in the analysis of stochastic optimization algorithms.

Next we impose a standard smoothness condition on the component functions.
\begin{assumption}
\label{assu:smooth_uc}
Each function $f_i : \mathbb{R}^d \to \mathbb{R}$ is Lipschitz continuously differentiable, i.e.,
\[
\|\nabla f_i(x) - \nabla f_i(y)\|
\le
L_{f,i}\|x-y\|,
\qquad \forall x,y\in\mathbb{R}^d.
\]

Consequently, the averaged function $f(x)$ satisfies
\[
\|\nabla f(x)-\nabla f(y)\|
\le
L_f\|x-y\|.
\]

At iteration $k$, the local constants satisfy
\[
L_{f,i}^k \le L_{f,i}.
\]
\end{assumption}

To control the stochasticity of the gradient estimators, we impose the strong growth condition.

\begin{assumption}
\label{assu:sgc_uc}
At any iteration $k$, the stochastic gradient estimator $g_k$ satisfies
\[
\mathbb{E}\!\left[\|g_k\|^{2}\right]
\le
\rho\,\|\nabla f(x_k)\|^{2},
\]
where $\rho>0$ is a constant.
\end{assumption}

Assumption~\ref{assu:sgc_uc} has been widely used
in the analysis of stochastic gradient methods. It has been shown to hold in various learning problems,
particularly in over-parameterized regimes encountered in modern machine learning models.

Under the above assumptions, we first establish several auxiliary results that will be used throughout the convergence analysis.

Recall from \eqref{subproblem 1} that the step $p_k$ solves the trust-region subproblem and admits the form 
\[p_k = -a_k g_k,\] 
where $g_k = \nabla f_i(x_k)$ and $a_k\in(0,1]$. Controlling the range of $a_k$ is crucial for establishing
sufficient descent at each successful iteration.

The following Lemma \ref{lem:range_ak_1} characterize the range of $a_k$.

\begin{lemma}
    \label{lem:range_ak_1}
    Suppose that Assumption \ref{assu:smooth_uc} holds. Then the parameter $a_{k}$ generated by the Algorithm \ref{alg:stochastic-tr} satisfies the following inequality
    \[
    \frac{2\left(1- c_0\right)}{L_{f,i}^k- c_0}\leq a_{k}\leq1.
    \]
\end{lemma}
\begin{proof}
By the smoothness of $f_i$, we have
\begin{align}
f_{i}\left(x_{k+1}\right) & \leq f_{i}\left(x_{k}\right)-\left\langle g_k,x_{k+1}-x_{k}\right\rangle +\frac{L_{f,i}^k}{2}\left\Vert x_{k+1}-x_{k}\right\Vert ^{2} \notag\\
 & =f_{i}\left(x_{k}\right)-a_{k}\left\Vert g_k\right\Vert ^{2}+\frac{L_{f,i}^k a_{k}^{2}}{2}\left\Vert g_k\right\Vert ^{2} \notag\\
 & =f_{i}\left(x_{k}\right)-a_{k}\left(1-\frac{L_{f,i}^ka_{k}}{2}\right)\left\Vert g_k\right\Vert ^{2}.
 \label{eq:smooth}
\end{align}
On the other hand, by the acceptance criterion in Algorithm~\ref{alg:stochastic-tr}, any successful iteration satisfies
\[f_i(x_k)-f_i(x_{k+1})
\ge
c_0\big(m_k(0)-m_k(p_k)\big).\]
Using the explicit form of $p_k$, this implies
\begin{equation}
\label{eq:accept}
    f_i(x_{k+1})\le f_i(x_k)-c_0 a_k\!\left(1-\frac{a_k}{2}\right)\!\|g_k\|^2.
\end{equation}
Combining \eqref{eq:smooth} and \eqref{eq:accept} yields
\[c_0\!\left(1-\frac{a_k}{2}\right)
\ge
1-\frac{L_{f,i}^k a_k}{2},\]
which is equivalent to
\[
\frac{L_{f,i}^k - c_0}{2}\, a_k \ge 1-c_0.
\]
Since $L_{f,i}^k>c_0$, we obtain
\[
a_{k}  \geq\frac{2\left(1- c_0\right)}{L_{f,i}^k - c_0}
\]

The upper bound of $a_k \le 1$ follows directly from the definition of $a_k$.
\end{proof}

Next, we begin to analyze the lower bound on trust region radius to make sure that there must exist infinitely many accepted steps.

\begin{lemma}
	\label{lem:penalty-model-accuracy}
	Suppose that Assumption \ref{assu:smooth_uc} holds. Then, there exists a constant \( M > 0 \) such that
	\[
	|m_k(p_k) - f_i(x_k + p_k)| \leq M \|p_k\|^2,
	\]
    where $M=L_{f,\max}$ and $L_{f,\max} = \max_{i}L_{f,i}$.
\end{lemma}

\begin{proof}
Since from Taylor's theorem we have that
\begin{equation*}
    f_{i}(x_k+p_k)=f_{i}(x_k)+g(x_k)^Tp_k+\int_{0}^{1} [g(x_k+tp_k)-g(x_k)]^Tp_k \, \mathrm{d}t,
\end{equation*}
for some $t\in(0,1)$, it follows from the definition of $m_k$ that
\begin{align*}
    |m_k(p_k)-f_{i}(x_k+p_k)|&=\left|\frac{1}{2}p_k^Tp_k-\int_{0}^{1} [g(x_k+tp_k)-g(x_k)]^Tp_k \,\mathrm{d}t\right|\\
    &\leq \frac{L_{f,\max}}{2}\lVert p_k\rVert^2+\frac{L_{f,i}}{2}\lVert p_k\rVert^2\\
    &=M\lVert p_k\rVert^2\\
    &\leq M\Delta_k^2.
\end{align*}

where $M=L_{f,\max}$ and $L_{f,\max} = \max_{i}L_{f,i}$.

\end{proof}

To facilitate the subsequent analysis of both Algorithm~\ref{alg:stochastic-tr} and Algorithm~\ref{alg:penalty-tr}, we first establish a general Cauchy decrease result that holds for any symmetric matrix $H_k$.

\begin{lemma}\label{lem:CauchyPoint}
    The Cauchy point $p_k$ satisfies
    \begin{equation*}
    g_k^{T}p_{k}+\frac{1}{2}p_{k}^{T}H_kp_{k}\leq-\frac{1}{2}\left\Vert g_k\right\Vert _{2}\min\left\{ \Delta_{k},\frac{\left\Vert g_k\right\Vert _{2}}{\left\Vert H_k\right\Vert _{2}}\right\}.
    \end{equation*}
\end{lemma}
\begin{lemma}
    \label{lem:Ultimate Progress at Nonoptimal Points}
   In Algorithm \ref{alg:stochastic-tr}, if at iteration $k$ we have $g_k\neq 0$ and 
   \begin{equation*}
       \Delta_k\leq \frac{\lVert g_k\rVert}{rM}
   \end{equation*}
where $r=\frac{2}{1-c_2}$, then we have that $\Delta_{k+1}=\min(\nu\Delta_k,\bar{\Delta})$.
\end{lemma}
\begin{proof}
Since $r>2$, we have $rM>2M>\lVert H_k\rVert$, and thus 
\begin{equation*}
    \Delta_k<\lVert g_k\rVert /\lVert H_k\rVert,
\end{equation*}
which means
\begin{equation*}
    \min(\Delta_k,\lVert g_k\rVert /\lVert H_k\rVert)=\Delta_k.
\end{equation*}
Then we have that for all $\Delta_k\leq \frac{\lVert g_k\rVert}{rM}$,
\begin{align*}
    |r_k-1|&=\frac{|f_{i}(x_k+p_k)-m_k(p_k)|}{m_k(0)-m_k(p_k)}\\
    &\overset{(i)}{\leq} \frac{M\Delta_k^2}{\frac{1}{2}\lVert g_k \rVert  \cdot\min[\frac{\lVert g_k\rVert}{\lVert H_k\rVert},\Delta_k]}\\
    &\overset{(ii)}{<}\frac{M\Delta_k^2}{\frac{1}{2}\lVert g_k\rVert\Delta_k}\\
    &<\frac{M\frac{\lVert g_k\rVert}{rM}}{\frac{1}{2}\lVert g_k\rVert}\\
    &=\frac{2}{r}\\
    &=1-c_2,
\end{align*}
where $(i)$ is due to Lemma \ref{lem:CauchyPoint}, and $(ii)$ is from $\min(\Delta_k,\lVert g_k\rVert /\lVert H_k\rVert)=\Delta_k$.
This implies that $\mu_k>c_2$, and we have that $\Delta_{k+1}=\min(\nu\Delta_k,\bar{\Delta})$.
\end{proof}

\begin{corollary}
    \label{cor:Lower Bound on Trust Region Radius}
    If there exists a constant $\epsilon>0$ such that $\lVert g_k\rVert\geq \epsilon$ for all $k$, then
    \begin{equation*}
        \Delta_k\geq \frac{\epsilon}{4rM}
    \end{equation*}
\end{corollary}
\begin{proof}
    From Lemma \ref{lem:Ultimate Progress at Nonoptimal Points}, if $\Delta_k\leq \frac{\lVert g_k\rVert}{rM}$, the trust region radius will be increased. Thus, the trust region radius can never be reduced below $\frac{\epsilon}{4rM}$.
\end{proof}

Next lemma show that the ratio of unsuccessful iterations to successful ones remains bounded. Define
\begin{equation*}
    S=\{k\in \mathbb{N}|r_k>c_0\}, \quad U=\{k\in \mathbb{N}|r_k\leq c_0\}.
\end{equation*}

\begin{lemma}
    \label{the radio of unsuccessful iterations to successful ones}
    Let \(\#S_K\) and \(\#U_K\) denote the number of successful and unsuccessful iterations within the first \(K\) iterations, respectively. If there exists a constant $\epsilon>0$ such that $\lVert g_k\rVert\geq \epsilon$ for all $k\in [K]$, then  
\[
\frac{\# U_K}{\# S_K} \leq log\frac{4rM\bar{\Delta}}{\epsilon}/log\nu.
\]  
\end{lemma}
\begin{proof}
Consider the scenario following a successful iteration. Suppose that after this successful iteration, the radius is updated to \(\Delta_{k+1} = \hat{\Delta}\). Then, there may follow a series of unsuccessful iterations.

From Algorithm \ref{alg:stochastic-tr}, we know that the radius of the trust region during these iterations cannot exceed \(\bar{\Delta}\), and the successful iteration ensures that the radius is at least \(\hat{\Delta}\). Therefore, after \(m\) unsuccessful iterations, we can derive the following inequality by applying Corollary \ref{cor:Lower Bound on Trust Region Radius}:

\[
\left(\frac{1}{\nu}\right)^m \hat{\Delta} \geq \frac{\epsilon}{4rM},
\]
which simplifies to:
\[
m \leq \frac{\log \frac{4rM\hat{\Delta}}{\epsilon}}{\log \nu}.
\]
Since \(\hat{\Delta}\) is bounded above by \(\bar{\Delta}\), we obtain the upper bound:
\[
m \leq \frac{\log \frac{4rM\bar{\Delta}}{\epsilon}}{\log \nu}.
\]

Thus, after at most 
$\log\frac{4rM\bar{\Delta}}{\epsilon}/\log\nu$
consecutive unsuccessful iterations, there must be a successful iteration.
Therefore, we can bound the ratio of unsuccessful iterations to successful iterations as follows:
\[
\frac{\# U_K}{\# S_K} \leq \frac{\log \frac{4rM\bar{\Delta}}{\epsilon}}{\log \nu}.
\]
\end{proof}

This lemma states that as the number of iterations grows, the number of unsuccessful iterations relative to successful ones remains bounded by a constant ratio. 

We now analyze the convergence rate of the first-order stochastic trust region method for non-convex objectives.
We show that the method achieves an iteration complexity of $O(\epsilon^{-2}\log(1/\epsilon))$, provided that the step parameter $a_k$ is restricted to the interva $a_{k}\in\left[a_{\min},a_{\max}\right]$.
Here, $a_{\min}=\frac{2\left(1- c_0\right)}{L_{\max}- c_0}$ 
is given in Lemma~\ref{lem:range_ak_1}, while $a_{\max}$ will be specified later.

\begin{theorem}\label{thm:first_order}
    Suppose that Assumption \ref{assu:unbiased_uc}, Assumption \ref{assu:smooth_uc} and Assumption \ref{assu:sgc_uc} hold. Algorithm \ref{alg:stochastic-tr} achieves the rate:
    \[
\min_{k \in [K-1]} \mathbb{E} \left[ \|\nabla f(x_k)\|^2 \right] \leq \frac{2\delta}{K} \left( 1 + \frac{\log \frac{4rM\bar{\Delta}}{\epsilon}}{\log \nu} \right) \mathbb{E} [f(x_0) - f(x^*)].
\]
    where $\delta=1/(-\rho La_{\max}^{2}+\left(1-\rho\right)a_{\max}+\left(1+\rho\right)\frac{2\left(1- c_0\right)}{L_{\max}- c_0}).$ Thus, for some $K=O(\frac{1}{\epsilon}\log\frac{1}{\epsilon})$ one finds that
\[
\min_{k \in [K-1]} \mathbb{E} \left[ \|\nabla f(x_k)\|^2 \right] \leq \epsilon.
\]
\end{theorem}
\begin{proof}
Starting from $L$-smoothness of $f$:
\begin{align*}
f\left(x_{k+1}\right)-f\left(x_{k}\right) & \leq\left\langle \nabla f\left(x_{k}\right),x_{k+1}-x_{k}\right\rangle +\frac{L_f}{2}\left\Vert x_{k+1}-x_{k}\right\Vert ^{2}\\
 & =-a_{k}\left\langle \nabla f\left(x_{k}\right),g_k\right\rangle +\frac{L_fa_{k}^{2}}{2}\left\Vert g_k\right\Vert ^{2}\\
 & =\frac{a_{k}}{2}\left(\left\Vert \nabla f\left(x_{k}\right)-g_k\right\Vert ^{2}-\left\Vert \nabla f\left(x_{k}\right)\right\Vert ^{2}-\left\Vert g_k\right\Vert ^{2}\right)+\frac{L_fa_{k}^{2}}{2}\left\Vert g_k\right\Vert ^{2}\\
\Longrightarrow2\left(f\left(x_{k+1}\right)-f\left(x_{k}\right)\right) & \leq a_{k}\left\Vert \nabla f\left(x_{k}\right)-g_k\right\Vert ^{2}-a_{k}\left(\left\Vert \nabla f\left(x_{k}\right)\right\Vert ^{2}+\left\Vert g_k\right\Vert ^{2}\right)+L_fa_{k}^{2}\left\Vert g_k\right\Vert ^{2}.
\end{align*}
Let $L_{\max}=\max_i L_i$. Then, if the $k+1$-th iteration is unsuccessful ($r_{k} \leq c_{0}$), we denote that $a_{k} = 0$. Conversely, if the $k+1$-th iteration is successful ($r_{k} > c_{0}$),
Lemma \ref{lem:range_ak_1} guarantees that $a_{\min}=\min\{\frac{2\left(1- c_0\right)}{L_{\max}- c_0},a_{\max}\}\leq \min\{\frac{2\left(1- c_0\right)}{L_{ik}- c_0},a_{\max}\}\leq a_k\leq a_{\max}$. Thus, for any successful iteration $k+1$, using this property and taking expectations with respect to $g_k$,
\begin{align*}
2\mathbb{E}\left[f\left(x_{k+1}\right)-f\left(x_{k}\right)\right]\leq & a_{\max}\mathbb{E}\left[\left\Vert \nabla f\left(x_{k}\right)-g_k\right\Vert ^{2}\right]-a_{\min}\left\Vert \nabla f\left(x_{k}\right)\right\Vert ^{2}-a_{\min}\mathbb{E}\left[\left\Vert g_k\right\Vert ^{2}\right]\\
 & +L_fa_{\max}^{2}\mathbb{E}\left[\left\Vert g_k\right\Vert ^{2}\right]\\
= & \left(a_{\max}-a_{\min}+L_fa_{\max}^{2}\right)\mathbb{E}\left[\left\Vert g_k\right\Vert ^{2}\right]-\left(a_{\max}+a_{\min}\right)\left\Vert \nabla f\left(x_{k}\right)\right\Vert ^{2}\\
\leq & \left(a_{\max}-a_{\min}+L_fa_{\max}^{2}\right)\rho\left\Vert \nabla f\left(x_{k}\right)\right\Vert ^{2}-\left(a_{\max}+a_{\min}\right)\left\Vert \nabla f\left(x_{k}\right)\right\Vert ^{2}\\
= & \left(\left(a_{\max}-a_{\min}+L_fa_{\max}^{2}\right)\rho-\left(a_{\max}+a_{\min}\right)\right)\left\Vert \nabla f\left(x_{k}\right)\right\Vert ^{2}.
\end{align*}
Suppose $\delta^{-1}=\left(a_{\max}+a_{\min}\right)-\left(a_{\max}-a_{\min}+L_fa_{\max}^{2}\right)\rho>0$, the gradient $\left\Vert \nabla f\left(x_{k}\right)\right\Vert ^{2}$ can be bounded as
\[
\left\Vert \nabla f\left(x_{k}\right)\right\Vert ^{2}\leq 2\delta\mathbb{E}\left[f\left(x_{k}\right)-f\left(x_{k+1}\right)\right].
\]
Note that for all the iteration that not successful, we have
$f(x_k) - f(x_{k+1}) = 0 $, $\forall k \in U_{K}$ due to $a_k=0$. Thus, for all the successful iterations, denoted as $S_k$, we have
\[
\frac{1}{K} \sum_{k=0}^{K-1} \|\nabla f(x_k)\|^2 \leq\frac{2\delta}{\# S_K} \sum_{k \in S} f(x_k) - f(x_{k+1}).
\]
Taking expectations,
\[
\frac{1}{K} \sum_{k=0}^{K-1} \mathbb{E} \left[ \|\nabla f(x_k)\|^2 \right] \leq\frac{2\delta}{\# S_K} \sum_{k \in S} \mathbb{E} [f(x_k) - f(x_{k+1})].
\]

From lemma \ref{the radio of unsuccessful iterations to successful ones}, we derive a lower bound on the number of successful iterations:

\[
\# S_K \geq \frac{K}{1 + \frac{\log \frac{4rM\bar{\Delta}}{\epsilon}}{\log \nu}}.
\]

Substituting \( \# S_K \geq \frac{K}{1 + \frac{\log \frac{4rM\bar{\Delta}}{\epsilon}}{\log \nu}} \):

\[
\frac{1}{K} \sum_{k=0}^{K-1} \mathbb{E} \left[ \|\nabla f(x_k)\|^2 \right] \leq \frac{2\delta}{K} \left( 1 + \frac{\log \frac{4rM\bar{\Delta}}{\epsilon}}{\log \nu} \right) \mathbb{E} [f(x_0) - f(x^*)].
\]
Thus,
\[
\min_{k \in [K-1]} \mathbb{E} \left[ \|\nabla f(x_k)\|^2 \right] \leq \frac{2\delta}{K} \left( 1 + \frac{\log \frac{4rM\bar{\Delta}}{\epsilon}}{\log \nu} \right) \mathbb{E} [f(x_0) - f(x^*)].
\]
For some $K=O(\frac{1}{\epsilon}\log\frac{1}{\epsilon})$ one finds that
\[
\min_{k \in [K-1]} \mathbb{E} \left[ \|\nabla f(x_k)\|^2 \right] \leq \epsilon.
\]
It remains to choose $a_{\max}$ such that $\delta>0$ holds. When
$a_{\max}\leq\frac{2\left(1- c_0\right)}{L_{\max}- c_0}$, we have $a_{\min}=a_{\max}=\frac{2\left(1- c_0\right)}{L_{\max}- c_0}$
and

\begin{align*}
\frac{1}{\delta} & =\left(a_{\max}+a_{\min}\right)-\left(a_{\max}-a_{\min}+L_fa_{\max}^{2}\right)\\
 & =2a_{\max}-L_f\rho a_{\max}^{2}>0\\
\Longrightarrow a_{\max} & <\frac{2}{L_f\rho}.
\end{align*}

When $a_{\max}>\frac{2\left(1- c_0\right)}{L_{\max}- c_0}$, we have

\begin{align*}
\frac{1}{\delta} & =\left(a_{\max}+\frac{2\left(1- c_0\right)}{L_{\max}- c_0}\right)-\left(a_{\max}-\frac{2\left(1- c_0\right)}{L_{\max}- c_0}+L_fa_{\max}^{2}\right)\rho\\
 & =-\rho L_fa_{\max}^{2}+\left(1-\rho\right)a_{\max}+\left(1+\rho\right)\frac{2\left(1- c_0\right)}{L_{\max}- c_0}.
\end{align*}

This is a concave quadratic in $a_{\max}$ and is strictly positive
when
\[
a_{\max}\in\left(0,\frac{\left(1-\rho\right)+\sqrt{\left(1-\rho\right)^{2}+4\rho L_f\left(1+\rho\right)\frac{2\left(1- c_0\right)}{L_{\max}- c_0}}}{2\rho L_f}\right).
\]
To avoid contradiction with the case assumption $\frac{2\left(1- c_0\right)}{L_{\max}- c_0}<a_{\max}$,
we require
\begin{align*}
\frac{\left(1-\rho\right)+\sqrt{\left(1-\rho\right)^{2}+4\rho L_f\left(1+\rho\right)\frac{2\left(1- c_0\right)}{L_{\max}- c_0}}}{2\rho L_f} & >\frac{2\left(1- c_0\right)}{L_{\max}- c_0}\\
\Longrightarrow\frac{1- c_0}{L_{\max}- c_0} & <\frac{1}{\rho L_f}.
\end{align*}
Thus, by choosing $ c_0$ such that $\frac{1- c_0}{L_{\max}- c_0}<\frac{1}{\rho L_f},$
we have a similar requirement for $a_{\max}$,
\begin{align*}
a_{\max} & <\frac{\left(1-\rho\right)+\sqrt{\left(1-\rho\right)^{2}+4\rho L_f\left(1+\rho\right)\frac{2\left(1- c_0\right)}{L_{\max}- c_0}}}{2\rho L_f}\\
 & <\frac{\left(1-\rho\right)+\sqrt{\left(1-\rho\right)^{2}+8\left(1+\rho\right)}}{2\rho L_f}\\
 & =\frac{1-\rho+3+\rho}{2\rho L_f}=\frac{2}{\rho L_f}.
\end{align*}
\end{proof}

\subsection{Convergence Analysis for the Equality-Constrained Problem}
We now analyze Algorithm~\ref{alg:penalty-tr} for solving the equality-constrained problem. 
The algorithm adopts the same stochastic trust-region framework as 
Algorithm~\ref{alg:stochastic-tr}, but constructs the quadratic model using the penalized objective
\[
\phi_i(x) = f_i(x) + \frac{\mu}{2}\|c(x)\|^2 .
\]

Because of this structural similarity, most analytical results derived for the unconstrained case 
remain valid after replacing $f_i$ with $\phi_i$. In particular, the Cauchy decrease result 
(Lemma~\ref{lem:CauchyPoint}), the lower bound on the trust-region radius, and the bound on the ratio 
between unsuccessful and successful iterations continue to hold with only minor modifications. 
For brevity, we omit the detailed adjustments and focus on the additional assumptions and arguments 
required by the penalty formulation. Throughout this subsection we assume without loss of generality that $\varepsilon \le 1$.

Similar to the unconstrained case, we assume that the stochastic gradient estimator is unbiased and that the objective is sufficiently smooth.

\begin{assumption}
\label{assu:unbiased_con}
For all $k \in \mathbb{N}$, the stochastic gradient estimator $g_k$ satisfies:  $$\mathbb{E}_k[g_k] = \nabla \phi(x_k).$$
\end{assumption}

\begin{assumption}
\label{assu:smooth_con}
Each function $f_i : \mathbb{R}^d \to \mathbb{R}$ and the constraint mapping
$c : \mathbb{R}^d \to \mathbb{R}^m$ are Lipschitz continuously differentiable, i.e.,
\[
\|\nabla f_i(x) - \nabla f_i(y)\| \le L_{f,i}\|x-y\|,
\qquad
\|\nabla c(x)-\nabla c(y)\| \le L_c\|x-y\|.
\]
\end{assumption}
Under this assumption, the penalized function
\[
\phi_i(x)=f_i(x)+\frac{\mu}{2}\|c(x)\|^2
\]
is also Lipschitz continuously differentiable with constant $L_{\phi,i}$.
Moreover, the averaged penalized function $\phi$ satisfies
\[
\|\nabla\phi(x)-\nabla\phi(y)\|\le L_\phi\|x-y\|.
\]
At iteration $k$, the corresponding local constants satisfy
\[
L_{f,i}^k\le L_{f,i}, \qquad L_{\phi,i}^k\le L_{\phi,i}.
\]

To simplify the analysis, we assume that the sequence of iterates remains in a bounded region.

\begin{assumption}
\label{assu:compact}
    The sequence $\{x_k\}$ generated by the algorithm is contained in a compact set $\mathcal{X} \subset \mathbb{R}^d$; that is, there exists a bounded closed set $\mathcal{X}$ such that
\[
x_k \in \mathcal{X}, \quad \forall k \ge 0.
\]
\end{assumption}
To ensure regularity of the feasible set, we impose the linear independence constraint qualification.
\begin{assumption}
\label{assu:LICQ}
The linear independence constraint qualification (LICQ) holds at any feasible point $x$ of
problem~\eqref{eq:constraint}. That is, the Jacobian $J(x)$ has full row rank for all $x$ satisfying
$c(x)=0$, and
\[
\sigma_{\min}(J(x)) \ge \sigma_{\min} > 0.
\]
\end{assumption}

Assumption~\ref{assu:LICQ} imposes the linear independence
constraint qualification (LICQ) at all feasible points.
This condition ensures regularity of the feasible set and well-posedness
of the associated Lagrange multipliers, which is essential for the
analysis of equality-constrained trust-region methods.

Finally, we extend the strong growth condition to the stochastic min--max formulation arising from the
equality-constrained problem. Under mild regularity conditions, problem~\eqref{eq:constraint} is
equivalent to the following unconstrained min--max problem:
\[
\min_{x\in\mathbb{R}^d}\max_{\lambda\in\mathbb{R}^m} L(x,\lambda),
\]
where $\lambda \in \mathbb{R}^m$ denotes the Lagrange multiplier associated with the constraints, and $L(x,\lambda)=f(x)+\lambda^\top c(x)$ is the Lagrangian function.

In the stochastic setting, the objective admits the representation
$f(x)=\mathbb{E}_{\xi\sim P}[f(x;\xi)]$, where $\xi$ denotes a random sample drawn from the distribution $P$. And the Lagrangian can be written as
\[
L(x,\lambda)=\mathbb{E}_{\xi\sim P}[L(x,\lambda;\xi)],
\qquad
L(x,\lambda;\xi)=f(x;\xi)+\lambda^\top c(x).
\]

\begin{assumption}
\label{assu:sgc_con}
There exists a constant $\tau > 1$ such that, for any $\lambda \in \mathbb{R}^m$,
\[
\mathbb{E}_\xi \bigl[ \|\nabla_x L(x,\lambda;\xi)\|^2 \bigr]
\le
\tau \|\nabla_x L(x,\lambda)\|^2.
\]
\end{assumption}

Assumption~\ref{assu:sgc_con} extends the strong growth condition to the stochastic Lagrangian in the min--max formulation, controlling the variance of the stochastic gradients with respect to the primal variable, and coincides with the strong growth condition considered in \cite{wang2023zeroth}.

Under the above assumptions, we next characterize the range of the parameter $a_k$ generated by Algorithm~\ref{alg:penalty-tr}.

\begin{lemma}
\label{lem:range_ak_2}
Suppose that Assumptions~\ref{assu:smooth_con} and 
\ref{assu:compact} hold, and let 
$\mu = 1/\varepsilon$ for some small $\varepsilon>0$.
Let
\[
C_c := \sup_{x\in\mathcal X} \|\nabla c(x)\| < \infty
\]
and define
\[
e := \max\{L_{\phi,i},\, 1+\mu C_c^2\}.
\]
Then the parameter $a_k$ generated by 
Algorithm~\ref{alg:penalty-tr} satisfies
\[
a_k
\ge
\frac{2(1-c_0)}
{\,e-c_0(1+\mu C_c^2)\,}.
\]
\end{lemma}
\begin{proof}
The proof follows the same argument as 
Lemma~\ref{lem:range_ak_1} applied to the penalized objective
\[
\phi_i(x)=f_i(x)+\frac{\mu}{2}\|c(x)\|^2.
\]

Since $c$ is continuously differentiable and the iterates 
$\{x_k\}$ remain in the compact set $\mathcal X$ 
by Assumption~\ref{assu:compact}, 
there exists $C_c>0$ such that
\[
\|\nabla c(x_k)\| \le C_c, \quad \forall k.
\]

By the smoothness of $\phi_i$, we obtain
\[
\phi_i(x_{k+1})
\le
\phi_i(x_k)
- a_k \|g_k\|^2
+ \frac{e}{2} a_k^2 \|g_k\|^2,
\]
where
\[
e=\max\{L_{\phi,i},\,1+\mu C_c^2\}.
\]

Using the acceptance criterion of 
Algorithm~\ref{alg:penalty-tr} and
\[
\|\nabla c(x_k)g_k\|^2
\le
C_c^2 \|g_k\|^2,
\]
we obtain
\[
(c_0-1)
+
\frac{a_k}{2}
\left(
e
-
c_0(1+\mu C_c^2)
\right)
\ge 0.
\]

Since 
\(
e>c_0(1+\mu C_c^2),
\)
the desired bound follows.
\end{proof}

\begin{lemma}
\label{lem:penalty-model-accuracy-2}
Suppose that Assumption~\ref{assu:smooth_con} holds. 
Then there exists a constant $M_{\mu}>0$ such that
\[
|m_k^{\phi}(p_k)-\phi_i(x_k+p_k)|
\le
M_{\mu}\|p_k\|^2,
\]
where
\[
M_{\mu}
=
\frac{1}{2}
\left(
L_{\phi,\max}
+
L_{f,\max}
+
\mu L_c
\right).
\]
\end{lemma}
\begin{proof}
Repeating the proof of Lemma~\ref{lem:penalty-model-accuracy}
for the function $\phi_i$ 
and using the Lipschitz bound on $\nabla\phi_i$ 
yields the result.
\end{proof}

To analyze the convergence of stochastic algorithms for the constrained
problem~\eqref{eq:constraint}, it is necessary to specify an appropriate
notion of approximate stationarity.
Unlike unconstrained optimization, where $\|\nabla f(x)\|$ serves as a
natural measure of optimality, constrained problems are typically
characterized by the Karush--Kuhn--Tucker (KKT) conditions.
Since stochastic methods can only be expected to achieve approximate
optimality, we adopt the standard notion of an $\varepsilon$-KKT point.

\begin{definition}[$\varepsilon$-KKT point]
Let $\varepsilon > 0$. 
A point $x \in \mathbb{R}^d$ is called an $\varepsilon$-KKT point of
problem~\eqref{eq:constraint} if there exists a Lagrange multiplier
$\lambda \in \mathbb{R}^m$ such that
\[
\|\nabla f(x) + \nabla c(x)^\top \lambda\| \le \varepsilon,
\qquad
\|c(x)\| \le \varepsilon.
\]
\end{definition}

\begin{theorem}
\label{thm:penalty-first-order} Suppose that Assumption \ref{assu:unbiased_con}, Assumption \ref{assu:smooth_con} Assumption \ref{assu:compact} and Assumption \ref{assu:sgc_con} hold. The algorithm operates as in Algorithm \ref{alg:penalty-tr}
with adaptive stepsizes $a_{k}$, the penalty parameter is chosen
as $\mu=\frac{1}{\epsilon}$ for some small $\epsilon>0$ to guarantee feasibility accuracy.

Then Algorithm \ref{alg:penalty-tr} achieves the rate: 
\[
\min_{k\in[K-1]}\mathbb{E}\left[\|\nabla\phi(x_{k})\|^{2}\right]\leq\frac{2\delta}{K}\left(1+\frac{\log\frac{4rM_{\mu}\bar{\Delta}}{\epsilon}}{\log\nu}\right)\mathbb{E}[\phi(x_{0})-\phi(x^{*})],
\]
where 
\[
\delta=1/(-\tau La_{\max}^{2}+(1-\tau)a_{\max}+(1+\tau)\cdot\frac{2(1-c_0)}
{\,e-c_0(1+\mu C_c^2)\,}).
\]
Hence, to guarantee $\mathbb{E}[\|\nabla\phi(x_{k})\|^{2}]\leq\epsilon$,
it suffices to run 
\[
K=O\left(\frac{1}{\epsilon^{2}}\log\frac{1}{\epsilon}\right)
\]
iterations.
\end{theorem}
\begin{proof}
The proof follows the same line of argument as 
Theorem~\ref{thm:first_order}, applied to the penalized objective
\(
\phi(x)=f(x)+\frac{\mu}{2}\|c(x)\|^{2}.
\)
From standard smoothness theory and the update $x_{k+1}=x_{k}-a_{k}g_{k}$, we obtain 
\[
\phi(x_{k+1})-\phi(x_{k}) \leq\frac{a_{k}}{2}\left(\|\nabla\phi(x_{k})-g_{k}\|^{2}-\|\nabla\phi(x_{k})\|^{2}-\|g_{k}\|^{2}\right)+\frac{L_{\phi}a_{k}^{2}}{2}\|g_{k}\|^{2}.
\]
As in the original analysis, for unsuccessful iterations, $a_k=0$. For successful iterations, the step parameter satisfies $a_{\min}\leq a_{k}\leq a_{\max}$,
where 
\[a_{\min}:=\min\left\{ \frac{2(1-c_0)}
{\,e-c_0(1+\mu C_c^2)\,},a_{\max}\right\} .\]
Taking expectations for successful iterations yields
\begin{align*}
2\mathbb{E}\left[\phi(x_{k+1})-\phi(x_{k})\right] & \leq (a_{\max}-a_{\min}+L_{\phi}a_{\max}^{2})\mathbb{E}[\|g_{k}\|^{2}]-(a_{\max}+a_{\min})\|\nabla\phi(x_{k})\|^{2}.
\end{align*}
Under Assumption \ref{assu:sgc_con}, let $\lambda=\mu c(x_k)$. The stochastic gradient satisfies
\[\mathbb{E}[\|g_k\|^2]\le \tau \|\nabla\phi(x_k)\|^2,\]
which implies
\[
2\mathbb{E}[\phi(x_{k+1})-\phi(x_{k})]\leq\left((a_{\max}-a_{\min}+L_{\phi}a_{\max}^{2})\tau-(a_{\max}+a_{\min})\right)\|\nabla\phi(x_{k})\|^{2}.
\]
Define 
\[\delta^{-1}:=(a_{\max}+a_{\min})-\tau(a_{\max}-a_{\min}+L_{\phi}a_{\max}^{2}).\]
To ensure $\delta>0$, the same quadratic argument as in 
Theorem~\ref{thm:first_order} yields the admissible restriction
\[
a_{\max}<\frac{2}{\tau L_{\phi}}.
\]
Then,
\[
\|\nabla\phi(x_{k})\|^{2}\leq2\delta\mathbb{E}[\phi(x_{k})-\phi(x_{k+1})].
\]
Summing over successful iterations and using the standard bound on the
number of unsuccessful steps gives
\[
\min_{k\in[K-1]}
\mathbb{E}\|\nabla\phi(x_k)\|^2
\le
\frac{2\delta}{K}
\left(
1+\frac{\log\frac{4rM_{\mu}\bar{\Delta}}{\epsilon}}{\log\nu}
\right)
\mathbb{E}[\phi(x_0)-\phi(x^*)].
\] 
Finally, since $\mu=\frac{1}{\epsilon}$,
\[
L_{\phi}
=
L_f+\mu L_c
=
O\!\left(\frac{1}{\epsilon}\right),
\]
which implies
\(
\delta=O(1/\epsilon)
\)
and therefore
\[
K
=
O\!\left(\frac{1}{\epsilon^{2}}\log\frac{1}{\epsilon}\right).
\]

\end{proof}

\begin{lemma}
	\label{lem:feasibility}
	Suppose that Assumption~\ref{assu:smooth_con}, Assumption~\ref{assu:compact} and Assumption~\ref{assu:LICQ} hold and the penalty parameter is set as $\mu = \frac{1}{\epsilon}$ for some small $\epsilon>0$.

	Then for any iterate $x_k$ satisfying $\|\nabla \phi(x_k)\| \leq \epsilon$, it holds that
	\[
	\|c(x_k)\| \leq \epsilon \cdot \frac{\epsilon + G}{\sigma_{\min}},
	\]
    where $G>0$ is a uniform bound of $\|\nabla f(x)\|$ over the compact set $\mathcal{X}$ guaranteed by Assumption~\ref{assu:compact}.
\end{lemma}

\begin{proof}
	By definition of the penalized objective,
	\[
\phi(x) = f(x) + \frac{\mu}{2}\|c(x)\|^2,
\qquad 
\nabla \phi(x) = \nabla f(x) + \mu J(x)^{\top} c(x).
	\]
Evaluating at $x_k$ and	rearranging yields
	\[
	J(x_k)^\top c(x_k) = \frac{1}{\mu} \left( \nabla \phi(x_k) - \nabla f(x_k) \right).
	\]
	Taking norms and using the triangle inequality gives
	\[
	\|J(x_k)^\top c(x_k)\| \leq \frac{1}{\mu} \left( \|\nabla \phi(x_k)\| + \|\nabla f(x_k)\| \right).
	\]
    Under Assumption~\ref{assu:LICQ}, the Jacobian $J(x_k)$ has full row rank, and its smallest singular value satisfies $\sigma_{\min}>0$.
    Consequently,
	\[
	\|J(x_k)^\top c(x_k)\| \geq \sigma_{\min} \cdot \|c(x_k)\|.
	\]
	Combining the two inequalities yields
    \[\|c(x_k)\|
    \le \frac{1}{\mu\,\sigma_{\min}}
    \bigl(\|\nabla \phi(x_k)\| + \|\nabla f(x_k)\|\bigr).\]
    Using Assumption~\ref{assu:compact}, since $\nabla f$ is continuous and the iterates remain in a compact set, there exists a constant $G>0$ such that
\[
\|\nabla f(x_k)\| \le G.
\]
Together with the condition $\|\nabla \phi(x_k)\| \le \epsilon$, we obtain
\[\|c(x_k)\|
    \le \frac{1}{\mu\,\sigma_{\min}}(\epsilon + G)
    = \epsilon\,\frac{\epsilon + G}{\sigma_{\min}},\]
where the last equality uses $\mu = 1/\epsilon.$

This concludes the proof.
\end{proof}

\begin{corollary}
\label{cor:kkt}
Suppose that the assumptions of Theorem~\ref{thm:penalty-first-order} hold and let $\varepsilon \le 1$. 
Then for any iterate $x_k$ satisfying 
\[
\|\nabla \phi(x_k)\| \le \varepsilon,
\]
there exists a multiplier $\lambda_k := \mu c(x_k)$ such that
\[
\|\nabla f(x_k) + \nabla c(x_k)^\top \lambda_k\| \le \varepsilon
\quad \text{and} \quad
\|c(x_k)\| \le C \varepsilon,
\]
where
\[
C := \frac{G+1}{\sigma_{\min}}
\]
is a constant independent of $\varepsilon$.
Hence, $x_k$ is an $O(\varepsilon)$-approximate KKT point.
\end{corollary}

\section{Numerical Experiments}
We evaluate the proposed STR method in over-parameterized stochastic optimization problems, considering both unconstrained and constrained settings. In all cases, STR is evaluated against standard stochastic optimization baselines and demonstrates competitive performance and improved robustness.

\subsection{Unconstrained Setting: Multi-class Classification}

We consider unconstrained multi-class image classification on the CIFAR-10 dataset, a standard benchmark for evaluating stochastic optimization methods in over-parameterized and highly nonconvex settings. As the learning model, we use a ResNet-20 architecture, following the standard design of \cite{he2016deep}.

We benchmark STR against representative stochastic optimization methods:
(1) constant step-size SGD,
(2) stochastic line-search methods (SLS), and
(3) Adam,
with hyperparameters selected by grid search using comparable tuning effort across methods.

For reproducibility, the final hyperparameters used in our experiments are listed below:

\begin{itemize}
    \item \textbf{SGD:} learning rate $lr = 0.2$, momentum $= 0.0$, weight decay $= 0.0$, mini-batch size $=128$.
    
    \item \textbf{Adam:} learning rate $lr = 10^{-3}$, $\beta_1 = 0.9$, $\beta_2 = 0.999$, $\epsilon = 10^{-8}$, mini-batch size $=128$.
    
    \item \textbf{SLS:} $c = 0.05$, $\beta = 0.9$, maximum step size $lr_{\max} = 2.0$, weight decay $= 0.0$, mini-batch size $=128$.
    
    \item \textbf{STR:} initial trust-region radius $\Delta_0 = 8$, maximum radius $\Delta_{\max} = 80$,
    $c_0 = 0.05$, $c_1 = 0.10$, $c_2 = 0.50$, 
    $\nu_1 = 2.0$, $\nu_2 = 5.0$, 
    and $\delta_{\max} = 80$, mini-batch size $=128$.
\end{itemize}

Each experiment is repeated over multiple independent runs with different random initializations, and performance is reported using the mean and standard deviation across runs.

\begin{figure}[t]
\centering
\includegraphics[width=\textwidth]{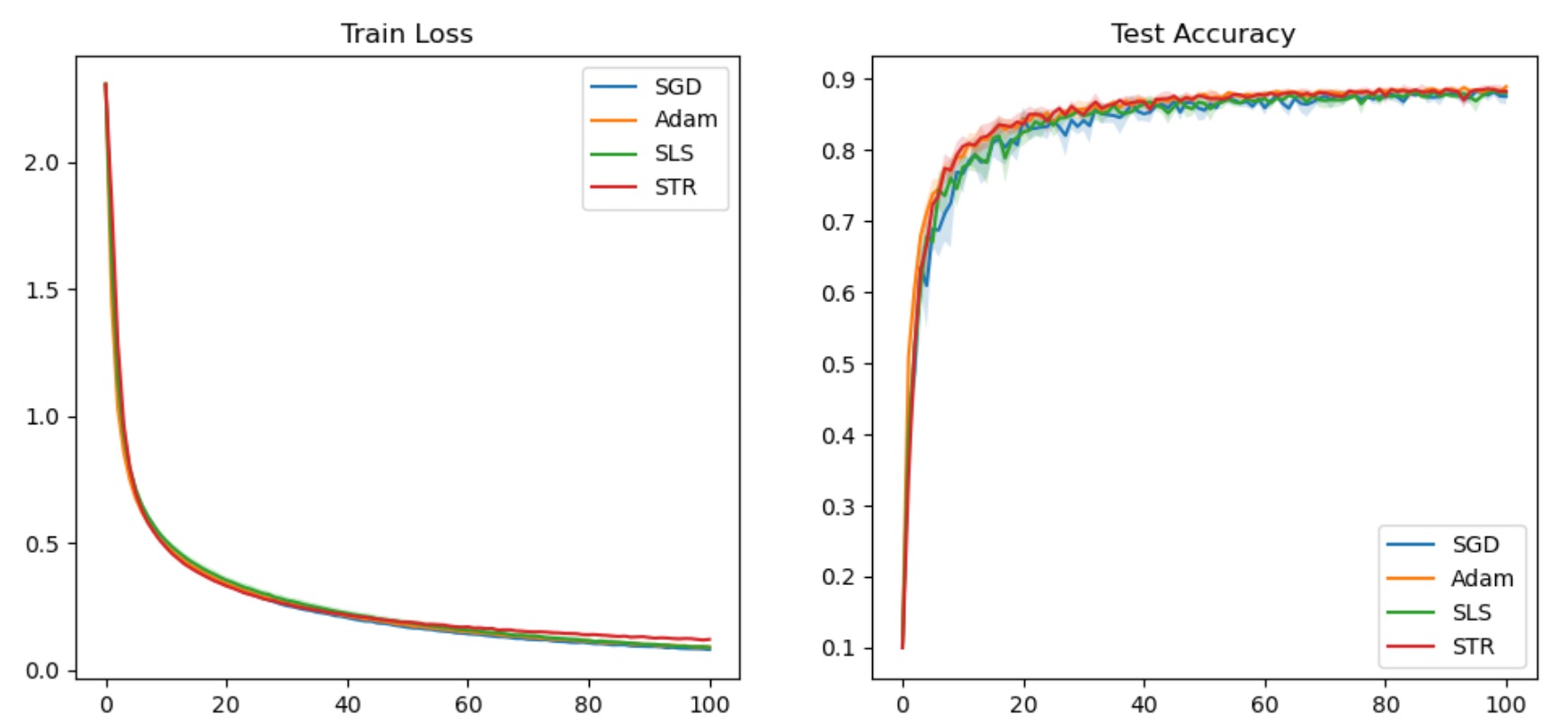}
\caption{Comparison of training loss and test accuracy for SGD, Adam, SLS, and STR on CIFAR-10.}
\label{fig:Multi-class}
\end{figure}

Figure~\ref{fig:Multi-class} shows the training loss and test accuracy curves. All methods exhibit similar convergence behavior and achieve nearly identical final test accuracy. STR matches the performance of the best baselines throughout training. While its training loss is slightly higher in later epochs, this conservative behavior does not affect generalization and reflects the stability of trust-region updates under stochastic noise. Overall, STR demonstrates competitive performance comparable to standard stochastic optimizers in unconstrained deep learning problems.

\subsection{Constrained Setting: Orthogonal Subspace Fitting}

We consider a constrained stochastic optimization problem arising in orthogonal subspace estimation. Given a data matrix $X \in \mathbb{R}^{d \times n}$ whose columns lie approximately in a $k$-dimensional subspace, we solve
\[
\min_{W \in \mathbb{R}^{d \times k}} \; \mathcal{L}(W) := \|(I - WW^\top)X\|_F^2 
\quad \text{s.t.} \quad W^\top W = I_k.
\]
This formulation corresponds to the classical orthogonal subspace fitting (or PCA) problem and has been widely studied in the context of constrained optimization on the Stiefel manifold~\cite{wen2013feasible}.

We convert the hard constraint into a quadratic penalty:
\[
\Phi_\mu(W) = \mathcal{L}(W) + \frac{\mu}{2} \|W^\top W - I_k\|_F^2,
\]
and apply the stochastic trust-region method described in Section 3 with penalty parameter $\mu = 1.0$.

Synthetic data are generated according to the spiked covariance model in~\cite{dar2020subspace}, with ambient dimension $d$ and intrinsic subspace dimension $k \ll d$.

We compare the proposed stochastic trust-region method with penalty (STR-P) against the following constrained stochastic optimization methods:
(1) Projected SGD with QR retraction (SGD+Proj),
(2) Riemannian gradient descent on the Stiefel manifold (RiemannianGD),
(3) Stochastic augmented Lagrangian method (AugLag),
with hyperparameters selected by grid search using comparable tuning effort across methods.

For reproducibility, the final hyperparameters used in our experiments are listed below:

\begin{itemize}
    \item \textbf{SGD+Proj:} learning rate $5\times 10^{-2}$, mini-batch size $32$.
    
    \item \textbf{RiemannianGD:} learning rate $5\times 10^{-2}$, mini-batch size $32$.
    
    \item \textbf{AugLag:} inner learning rate $0.01$, initial penalty parameter $\mu_0 = 0.1$, penalty growth factor $1.1$, inner epochs $10$, mini-batch size $32$, and damping parameter $\lambda_{\text{damp}} = 0.5$.
    
    \item \textbf{STR-P:} penalty parameter $\mu = 1.0$, initial trust-region radius $\Delta_0 = 0.2$, maximum radius $\Delta_{\max} = 5.0$, $c_0 = 0.05$, $c_1 = 0.1$, $c_2 = 0.9$, $\nu_1 = 1.5$, $\nu_2 = 0.5$ and mini-batch size $32$.
\end{itemize}
\begin{figure}[t]
\centering
\includegraphics[width=\textwidth]{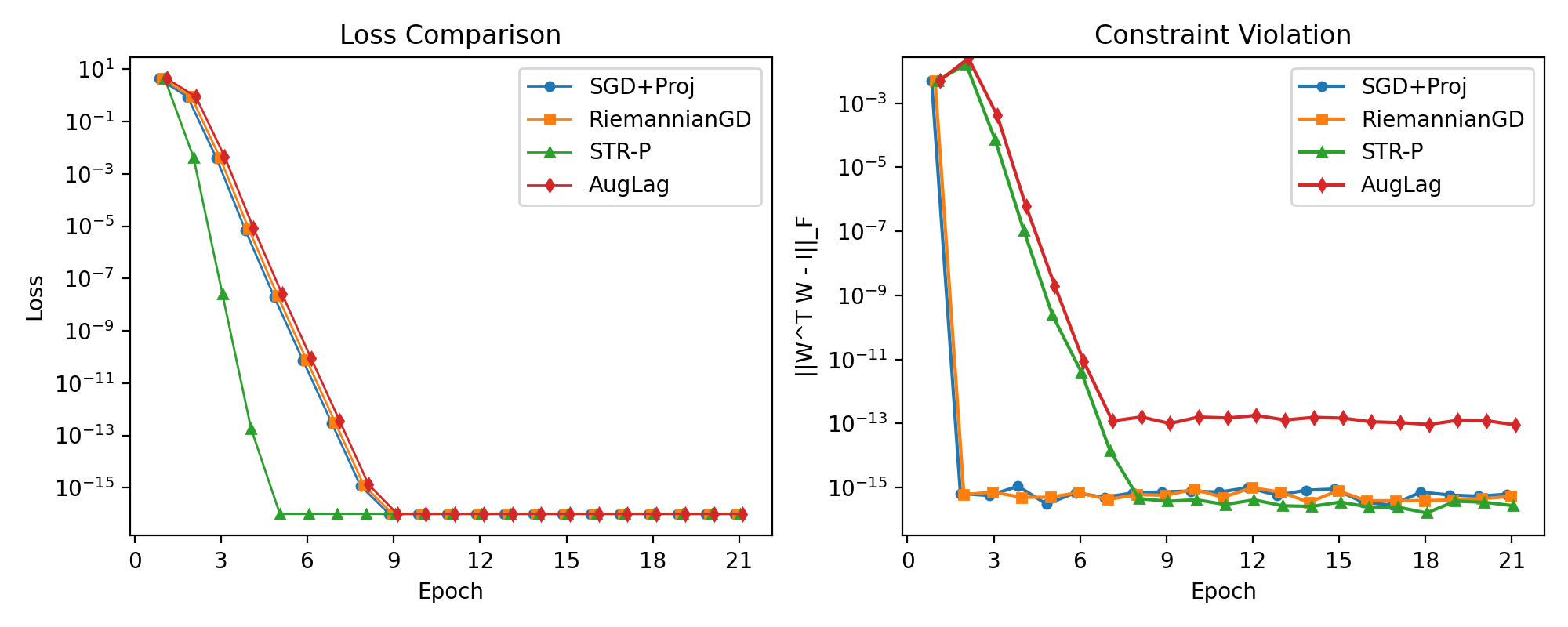}
\caption{Constrained orthogonal subspace fitting on synthetic spiked model data ($d=100$, $k=5$, $n=500$). Comparison of STR-P with projected SGD, Riemannian gradient descent, and stochastic augmented Lagrangian.}
\label{fig:pca_small}
\end{figure}

\begin{figure}[t]
\centering
\includegraphics[width=\textwidth]{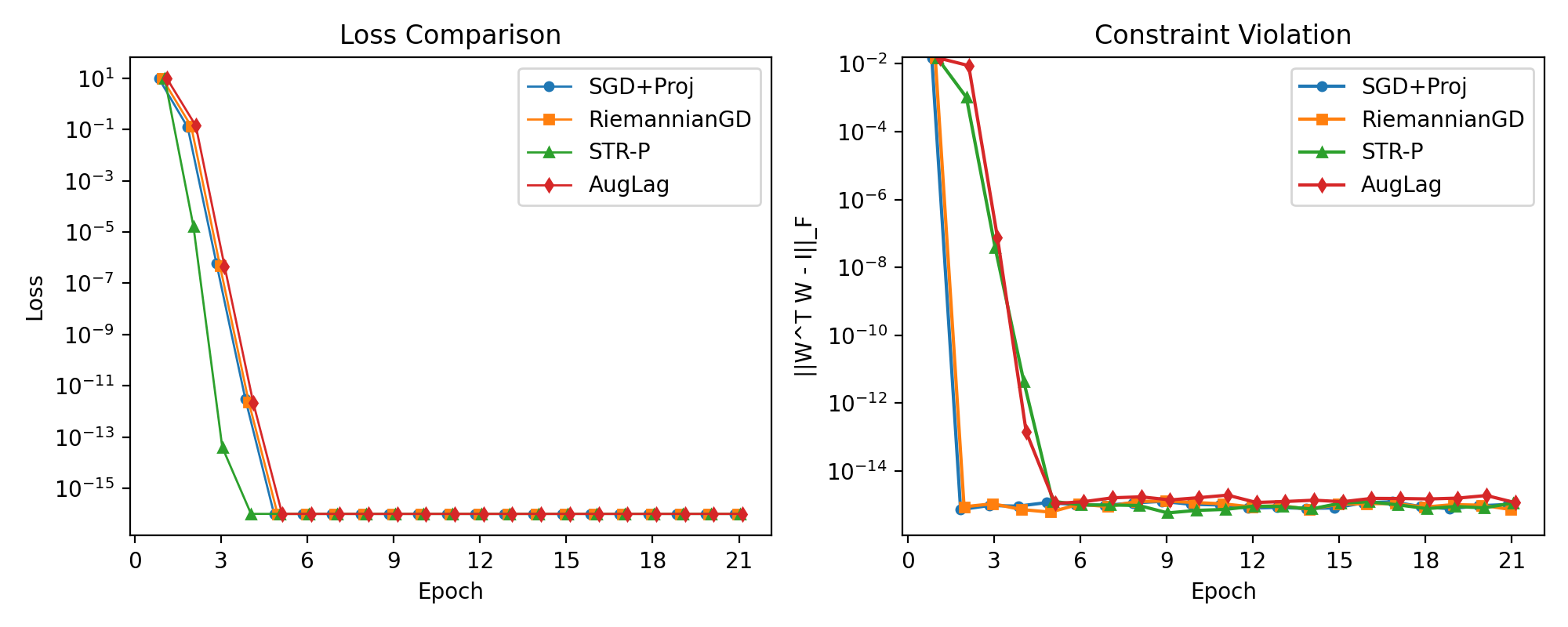}
\caption{Constrained orthogonal subspace fitting on synthetic spiked model data ($d=500$, $k=10$, $n=1000$). Comparison of STR-P with projected SGD, Riemannian gradient descent, and stochastic augmented Lagrangian.}
\label{fig:pca_large}
\end{figure}
Figure~\ref{fig:pca_small} and~\ref{fig:pca_large} compare different constrained stochastic optimization methods on the orthogonal subspace fitting problem for both moderate ($d=100,k=5,n=500$) and larger-scale ($d=500,k=10,n=1000$) settings. We report both the objective value and the constraint violation $\|W^\top W - I\|_F$. Projection-based methods (SGD+Proj and RiemannianGD) maintain feasibility at near machine precision throughout the optimization due to QR-based retraction at every iteration. In contrast, penalty-based methods (STR-P and AugLag) do not enforce feasibility explicitly and therefore exhibit a gradual reduction of the constraint violation. Despite this, STR-P consistently achieves faster objective reduction in the early and intermediate phases and drives the constraint violation to near machine precision within a small number of epochs. This behavior is consistent across both problem scales, demonstrating that the proposed method effectively balances fast convergence with accurate constraint satisfaction without requiring explicit projection at each step.

\section{Conclusion}
In this work, we proposed a unified framework for stochastic trust-region methods in both unconstrained and equality-constrained over-parameterized optimization problems. By leveraging interpolation and strong growth conditions, our first-order stochastic trust-region algorithm achieves an iteration complexity of $O(\epsilon^{-2} \log(1/\epsilon))$ for unconstrained problems, while the quadratic-penalty-based method attains an $\epsilon$-stationary point with $O(\epsilon^{-4} \log(1/\epsilon))$ iterations for constrained problems. The theoretical results are validated through numerical experiments on deep learning and orthogonally constrained subspace fitting tasks, demonstrating that stochastic trust-region methods can achieve stable and competitive performance without extensive manual step-size tuning. To the best of our knowledge, this is the first unified complexity analysis of stochastic trust-region methods under interpolation-type assumptions. Future work may explore extensions to second-order variants, improved model construction strategies, and adaptive penalty update mechanisms to further enhance efficiency and scalability.

\section*{Acknowledgement(s)}

The authors gratefully acknowledge Professor Xiao Wang from Sun Yat-sen University for her valuable guidance and helpful suggestions.

\bibliographystyle{tfs}
\bibliography{ref}

\end{document}